\documentclass[12pt, reqno]{amsart}
\usepackage{fullpage}
\usepackage{amssymb}
\usepackage{bbm}
\usepackage{mathrsfs}
\usepackage[cmtip,arrow]{xy}
\usepackage{pb-diagram,pb-xy}

\newcommand{\zz}{\mathbb{Z}}

\newcommand{\qq}{\mathbb{Q}}

\newcommand{\gf}{\mathbb{F}}
\newcommand{\vect}{\mathbf}
\newcommand{\ints}{\mathcal{O}}

\newcommand{\legendre}[2]{\left(\frac{#1}{#2}\right)}
\newcommand{\floor}[1]{\left\lfloor#1\right\rfloor}
\newcommand{\ceil}[1]{\left\lceil#1\right\rceil}

\newcommand{\dnd}{\nmid}
\newcommand{\isom}{\cong}

\newcommand{\Aut}{\mathrm{Aut}}

\newcommand{\ord}{\mathrm{ord}}
\newcommand{\ecbox}{\mathscr{B}}
\newcommand{\Gal}{\mathrm{Gal}}
\newcommand{\N}{\mathrm{N}}
\newcommand{\p}{\mathfrak{p}}

\newtheorem{prop}{Proposition}
\newtheorem{thm}{Theorem}
\newtheorem{lma}{Lemma}
\newtheorem{cor}{Corollary}

\newtheorem{conj}{Conjecture}
\theoremstyle{remark}
\newtheorem*{rmk}{Remark}

\DeclareMathAlphabet{\mathpzc}{OT1}{pzc}{m}{it}
\newcommand{\MK}{\mathpzc m_K}
\newcommand{\NK}{\mathpzc n_K}
\newcommand{\NA}{\mathpzc n_\mathcal{A}}

\begin{document}
\title{Average Frobenius distribution for elliptic curves defined over finite Galois extensions of 
the rationals\\ 
(appeared in \textit{Mathematical Proceedings of the Cambridge Philosophical Society})}

\author[James]{Kevin James}
\address[Kevin James]{
Department of Mathematical Sciences\\
Clemson University\\
Box 340975 Clemson, SC 29634-0975
}
\email{kevja@clemson.edu}
\urladdr{www.math.clemson.edu/~kevja}

\author[Smith]{Ethan Smith}
\address[Ethan Smith]{
Department of Mathematical Sciences\\
Michigan Technological University\\
1400 Townsend Drive\\
Houghton, MI 49931-1295
}
\email{ethans@mtu.edu}
\urladdr{www.math.mtu.edu/~ethans}

\begin{abstract}
Let $K$ be a fixed number field, assumed to be Galois over $\qq$.  Let $r$ and $f$ be fixed integers with $f$ positive.  
Given an elliptic curve $E$, defined over $K$, we consider the problem of
counting the number of degree $f$ prime ideals of $K$ with trace of Frobenius equal 
to $r$.  Except in the case $f=2$, we show that ``on average," the number of such prime ideals 
with norm less than or equal to $x$ satisfies an asymptotic identity that is in accordance with 
standard heuristics.  This work is related to the classical Lang-Trotter conjecture and extends 
the work of several authors. 
\end{abstract}

\maketitle

\section{Introduction}

We begin by reviewing the classical case.
Let $E$ be an elliptic curve defined over the rational field $\qq$.  For a prime $p$ where 
$E$ has good reduction, we let 
$a_p(E)$ denote the trace of the Frobenius morphism.  Then $a_p(E)=p+1-\#E(\gf_p)$, and 
it was shown by Hasse that 
$|a_p(E)|\le 2\sqrt p$.  See~\cite[p. 131]{Sil:1986}.
Given a fixed elliptic curve $E$ and a fixed integer $r$, the prime counting function
\begin{equation}
\pi_E^r(x):=\#\{p\le x: a_p(E)=r\}
\end{equation}
has received a great deal of attention.  
Deuring~\cite{Deu:1941} showed that if $E$ admits complex multiplication, 
then half the primes are of 
\textit{supersingular} reduction, i.e., $a_p(E)=0$.  
In addition, the distribution of the non-supersingular primes was 
explained by Hecke~\cite{Hec:1918, Hec:1920} for elliptic curves with 
complex multiplication.
For the remaining cases, 
we have the following conjecture of Lang and Trotter~\cite{LT:1976}.
\begin{conj}[Lang-Trotter]\label{LT Conj}
Let $E$ be a fixed elliptic curve defined over $\qq$, and let $r$ be a fixed integer.  In the case 
that $E$ has complex multiplication, also assume that $r\neq 0$.  There exists a 
constant $C_{E,r}$ such that
\begin{equation}
\pi_{E}^r(x)\sim C_{E,r}\frac{\sqrt x}{\log x}
\end{equation}
as $x\rightarrow\infty$.
The constant $C_{E,r}$ may be zero, in which case the asymptotic is interpreted to mean that 
there are only finitely many such primes.
\end{conj}

The theme of studying this conjecture ``on average" was initiated by 
Fouvry and Murty in~\cite{FM:1996} who considered the case when $r=0$.  
The density of complex multiplication curves is so little that they do not affect 
the asymptotic.
Their work was generalized by 
David and Pappalardi~\cite{DP:1999} who considered the remaining cases.
This was later improved by Baier~\cite{Bai:2007} 
who showed that the result also holds over a ``shorter average."
Finer averages have also been considered.  In~\cite{Jam:2004}, the first author considered 
the problem when the average was restricted to curves admitting a rational 3-torsion point.
Averages over families of elliptic curves with various prescribed torsion structures were 
considered in~\cite{BBIJ:2005}.

We now turn to the number field case.
Suppose that $K$ is a number field and $E$ is an elliptic curve defined over $K$.
Given a prime ideal $\p$ of the ring of integers $\ints_K$ where $E$ has good reduction,
we define the trace of Frobenius $a_\p(E)$ as before.
In particular, we have $a_\p(E)=\N\p+1-\#E(\ints_K/\p)$ and 
$|a_\p(E)|\le 2\sqrt{\N\p}=2p^{f/2}$.  Here, $\N\p:=\#(\ints_K/\p)=p^f$ is the norm of $\p$, $p$ is 
the unique rational prime lying below $\p$, and $f=\deg\p$ is the absolute degree of $\p$.
For a fixed elliptic curve $E$ and fixed integers $r$ and $f$, we define the prime counting 
function 
\begin{equation}
\pi_E^{r,f}(x):=\#\{\N\p\le x: a_\p(E)=r\text{ and } \deg\p=f\}.
\end{equation}
For elliptic curves defined over a number field $K$, the heuristics of Lang and 
Trotter~\cite{LT:1976} suggest the following more refined conjecture.
See~\cite{DP:2004} also.
\begin{conj}[Lang-Trotter for number fields]\label{LT for K}
Let $E$ be a fixed elliptic curve defined over $K$, and let $r$ be a fixed integer.
In the case that $E$ has complex multiplication, also assume that $r\ne 0$.
Let $f$ be a positive integer.
There exists a constant $\mathfrak{C}_{E,r,f}$ such that
\begin{equation}
\pi_E^{r,f}(x)\sim\mathfrak{C}_{E,r,f}\begin{cases}
\frac{\sqrt x}{\log x}& \text{if }f=1,\\
\log\log x& \text{if }f=2,\\
1& \text{if }f\ge 3
\end{cases}
\end{equation}
as $x\rightarrow\infty$.
The constant $\mathfrak{C}_{E,r,f}$ may be zero, in which case the asymptotic is interpreted 
to mean that there are only finitely many such primes.
\end{conj}

\begin{rmk}
For a fixed $f\ge 3$, 
we interpret the conjecture to say that there are only finitely 
many such primes.  In this case, the constant $\mathfrak{C}_{E,r,f}$ would 
necessarily be a nonnegative integer.
\end{rmk}

This conjecture too has been studied on average.  David and 
Pappalardi~\cite{DP:2004} considered the case when $K=\qq(i)$ and $f=2$.
A recent paper of Calkin, Faulkner, King, Penniston, 
and the first author~\cite{FJKP} extended their work to the setting of an 
arbitrary number field $K$ assumed to be Abelian over $\qq$.  In fact, the authors of~\cite{FJKP} considered any positive integer $f$ and obtained asymptotics in accordance with the 
conjecture.

The purpose of the present paper is to improve the work of~\cite{FJKP} in two ways.  
In the first place, we will relax the assumption that the number field $K$ is Abelian over 
$\qq$ provided that $f\neq 2$.   Instead, we will assume that $K$ is Galois over $\qq$.
The case $f=2$ remains somewhat elusive.  However, the authors are currently pursuing this case.
The second improvement is that we will consider a ``more general average" which will allow us to 
show that Conjecture~\ref{LT for K} still holds on average when averaging over a ``smaller"
set of elliptic curves.

\section{Statement of results}

For the remainder of the paper, we will assume that $K$ is a fixed number field.  In addition, we 
assume that the extension $K/\qq$ is Galois.  
We denote 
the degree of the extension by $\NK:=[K:\qq]$.
Recall that $\ints_K$ is a free $\zz$-module of rank $\NK$ and
let $\mathcal{B}=\{\gamma_j\}_{j=1}^{\NK}$ be a fixed integral basis for $\ints_K$.
We denote the coordinate map for the basis $\mathcal{B}$ by 
\begin{equation*}
[\cdot]_{\mathcal{B}}:\ints_K\stackrel{\sim}{\longrightarrow}
\bigoplus_{j=1}^{\NK}\zz=\zz^{\NK}.
\end{equation*}
Given two vectors $\vect a,\vect b\in\zz^{\NK}$, if each entry of $\vect a$ 
is less than or equal to the corresponding entry of $\vect b$, then we 
write $\vect a\le\vect b$.
If $\vect b\ge \vect 0$, 
then we define a ``box" of algebraic integers by
\begin{equation}\label{int box}
\mathcal{B}(\vect a, \vect b):=
\{\alpha\in\ints_K: \vect a-\vect b\le [\alpha]_\mathcal{B}\le \vect a+\vect b\}.
\end{equation}
For two algebraic integers $\alpha,\beta\in\ints_K$, we write $E_{\alpha,\beta}$ 
for the elliptic curve given by the model
\begin{equation*}
E_{\alpha,\beta}: Y^2=X^3+\alpha X+\beta.
\end{equation*}
Then for appropriate vectors, we define a ``box" of elliptic curves by
\begin{equation}
\ecbox:=\ecbox(\vect a_1,\vect b_1;\vect a_2,\vect b_2)
=\{E_{\alpha,\beta}: \alpha\in\mathcal{B}(\vect a_1,\vect b_1)\text{ and }
\beta\in\mathcal{B}(\vect a_2,\vect b_2)\}.
\end{equation}
To be more precise, this box should be thought of as a box of equations or models 
since the same elliptic curve may appear multiple times in $\ecbox$.
For $i=1,2$, let $b_{i,j}$ denote 
the $j$-th entry of $\vect b_i$.
Associated to box $\mathscr{B}$, we define the quantities
\begin{align}
\mathscr{V}(\ecbox)&=2^{2\NK}\prod_{j=1}^{\NK}b_{1,j}*b_{2,j},\\
\mathscr{V}_1(\ecbox)&=2^{\NK}\prod_{j=1}^{\NK}b_{1,j},\\
\mathscr{V}_2(\ecbox)&=2^{\NK}\prod_{j=1}^{\NK}b_{2,j},\\
\mathscr{V}_{\text{min}}(\ecbox)&=2\min_{1\le j\le\NK}\{b_{1,j},b_{2,j}\},
\end{align}
which give a description of the size of this box.
In particular, 
\begin{equation*}
\#\ecbox=\mathscr{V}(\ecbox)
+O\left(\mathscr{V}(\ecbox)/\mathscr{V}_{\min}(\ecbox)\right).
\end{equation*}

Recall that
\begin{equation}\label{pi1/2}
\pi_{1/2}(x):=\int_2^x\frac{dt}{2\sqrt t\log t}\sim\frac{\sqrt x}{\log x}.
\end{equation}
We are now ready to state the main results of this paper.

\begin{thm}\label{avg LT for K}
Let $r$ be a fixed integer.  Then, for any $\eta>0$,
\begin{equation*}
\frac{1}{\#\mathscr{B}}\sum_{E\in\mathscr{B}}\pi_{E}^{r,1}(x)=
\mathfrak{C}_{K,r,1}\pi_{1/2}(x)
+\mathcal{E}(x;\mathscr{B}),
\end{equation*}
where
\begin{equation*}
\mathcal{E}(x;\ecbox)\ll\frac{\sqrt x}{(\log x)^{1+\eta}}
+\frac{\sqrt x/\log x}{\mathscr{V}_{\min}(\ecbox)}
+\left(\frac{1}{\mathscr{V}_1(\ecbox)}+\frac{1}{\mathscr{V}_2(\ecbox)}\right)(x\log x)^{\NK}
+\frac{(x\log x)^{2\NK}}{\mathscr{V}(\ecbox)},
\end{equation*}
and $\mathfrak{C}_{K,r,1}$ is the constant defined by the absolutely convergent 
sum~\eqref{avg const} on page~\pageref{avg const}.
\end{thm}

As an immediate corollary of Theorem~\ref{avg LT for K}, we have the following.
\begin{cor}\label{main cor 1}
Let $\eta>0$, and let $r$ be a fixed integer.  Then
\begin{equation*}
\frac{1}{\#\ecbox}\sum_{E\in\ecbox}\pi_{E}^{r,1}(x)\sim
\mathfrak{C}_{K,r,1}\pi_{1/2}(x),
\end{equation*}
provided that the box $\ecbox$ satisfies the growth conditions: 
\begin{align*}
\mathscr{V}(\ecbox)&\gg x^{2\NK-1/2}(\log x)^{2\NK+1+\eta},\\
\mathscr{V}_1(\ecbox),\mathscr{V}_2(\ecbox)&\gg x^{\NK-1/2}(\log x)^{\NK+1+\eta},\\
\mathscr{V}_{\mathrm{min}}(\ecbox)&\gg(\log x)^{\eta}.
\end{align*}
\end{cor}
\begin{rmk}
The ``growth rate" of the box $\ecbox$ is much smaller than that of the 
corresponding box in~\cite{FJKP}.
\end{rmk}

We also consider the mean square error or how much the function $\pi_E^{r,1}(x)$ varies from the 
average $\mathfrak{C}_{K,r,1}\pi_{1/2}(x)$.
\begin{thm}\label{variance thm}
Let $\eta>0$.  Then
\begin{equation*}
\frac{1}{\#\mathscr{B}}\sum_{E\in\ecbox}
\left|\pi_E^{r,1}(x)-\mathfrak{C}_{K,r,1}\pi_{1/2}(x)\right|^2
\ll\frac{ x}{(\log x)^{2+\eta}},
\end{equation*}
provided that the box $\ecbox$ satisfies the growth conditions:
\begin{align*}
\mathscr{V}(\ecbox)&\gg x^{4\NK-1}(\log x)^{4\NK+2+\eta},\\
\mathscr{V}_1(\ecbox),\mathscr{V}_2(\ecbox)&\gg x^{2\NK-1}(\log x)^{2\NK+2+\eta},\\
\mathscr{V}_{\mathrm{min}}(\ecbox)&\gg(\log x)^{\eta}.
\end{align*}
\end{thm}
An application of the Tur\'an normal order method (see~\cite[Chapter 3]{CM:2006}) 
supplies us with the following corollary.
\begin{cor}\label{normal order}
Let $\delta,\eta>0$ be fixed with $\delta>2\eta$.
If $\ecbox$ satisfies the conditions of Theorem~\ref{variance thm}, then 
for all $E\in\ecbox$ with at most 
$O\left(\frac{\mathscr{V}(\ecbox)}{(\log x)^{\delta-2\eta}}\right)$ 
exceptions, we have
\begin{equation*}
\left|\pi_{E}^{r,1}(x)-\mathfrak{C}_{K,r,1}\pi_{1/2}(x)\right|<
\frac{\sqrt x}{(\log x)^{1+\eta}}.
\end{equation*}
\end{cor}
\begin{rmk}
Care should be taken with the interpretation of Corollary~\ref{normal order}.  
It would be easy to draw the conclusion that the average order constant 
$\mathfrak{C}_{K,r,1}$ is the correct constant $\mathfrak{C}_{E,r,1}$ for 
most elliptic curves $E$ defined over $K$.  
Although it is possible for this to happen for a given choice of $K$ and $r$, 
this is not implied by Corollary~\ref{normal order}.  
In fact, Corollary~\ref{normal order} does not even imply that there is 
one elliptic curve $E$ for which $\mathfrak{C}_{K,r,1}=\mathfrak{C}_{E,r,1}$.
The key fact to remember when interpreting Corollary~\ref{normal order} 
is that the curves appearing in the box $\ecbox$ depend on $x$.
Therefore, as $x$ changes, so might the exceptional set.  In fact, it is 
possible that for a given value of $r$, every elliptic curve defined over $K$ 
``eventually" enters the exceptional set.
\end{rmk}

For $f\ge 3$, we also have the following average order result for $\pi_E^{r,f}(x)$.
\begin{thm}\label{higher degree avg}
Let $r$ be a fixed integer.  If $f\ge 3$, then
\begin{equation*}
\lim_{x\rightarrow\infty}
\frac{1}{\#\ecbox}\sum_{E\in\ecbox}\pi_E^{r,f}(x)<\infty,
\end{equation*}
provided that $\mathscr{V}_{\mathrm{min}}(\ecbox)\gg x^{1/f}$.
\end{thm}
\begin{rmk}
The authors of~\cite{FJKP} state a version of this result under the assumption that 
$K/\qq$ is Abelian.  Although they only state their result for Abelian extensions, it turns out that 
their methods are sufficient to 
prove Theorem~\ref{higher degree avg} under the relaxed assumption that $K/\qq$ is 
Galois.  Therefore, we will omit the proof of Theorem~\ref{higher degree avg} and 
concentrate on the case $f=1$ in this paper.
\end{rmk}
\begin{rmk}
Unfortunately, using the techniques that we present here, 
the case $f=2$ is not so easily generalized to the setting of an arbitrary finite, 
normal extension of $\qq$.  See the discussion in Section~\ref{compute l-series avg}
preceding the proof of Proposition~\ref{weighted avg of l-series}.
However, the authors of this paper do believe it should be possible 
to make some progress towards such a result and are currently pursuing it for a certain class of
number fields possessing a non-Abelian Galois group over $\qq$.
\end{rmk}

\section{Acknowledgement}

The authors wish to thank Chantal David, Nathan Jones, and the anonymous referee for
helpful suggestions during the preparation of this article. 

\section{The average order constant}\label{avg order const sect}

In this section, we give a precise description of the constant $\mathfrak{C}_{K,r,1}$,
first as an infinite sum and then as an infinite product over primes.  This requires a 
considerable amount of additional notation.

Recall that $G=\Gal(K/\qq)$.   Let $[G,G]$ denote the commutator subgroup of $G$, 
and let $\mathcal A$ be the fixed field of $[G,G]$.  Then $\mathcal A/\qq$ is an 
Abelian extension of finite degree, which we denote by $\NA=[\mathcal A:\qq]$.
By the Kronecker-Weber Theorem~\cite[p.~210]{Lan:1994}, it follows that there is a 
smallest positive integer 
$\MK$ so that $\mathcal A\subseteq\qq(\zeta_{\MK})$.  Here, $\zeta_{\MK}$ 
is a primitive $\MK$-th root of unity.
It is well-known that $\Gal(\qq(\zeta_{\MK})/\qq)\isom(\zz/\MK\zz)^*$ with 
a natural choice of isomorphism.  
See~\cite[p.~11]{Was:1997} for example.
Let $G_{\MK}$\label{def GmK} denote the subgroup of $(\zz/\MK\zz)^*$ corresponding to 
$\Gal(\qq(\zeta_{\MK})/\mathcal A)$ under this isomorphism.

The constant $\mathfrak{C}_{K,r,1}$ is given by
\begin{equation}\label{avg const}
\mathfrak{C}_{K,r,1}:=\frac{2\NA}{\pi}
\left(\sum_{b\in G_{\MK}}\sum_{k=1}^\infty\sum_{n=1}^\infty
\frac{c_k^{r,b,\MK}(n)}{nk\varphi([\MK,nk^2])}\right),
\end{equation}
where
\begin{equation}\label{c fct defn}
c_k^{r,b,\MK}(n):=
\sum_{\substack{
a\in(\zz/4n\zz)\\ a\equiv 0,1\pmod{4}\\
(r^2-ak^2,4nk^2)=4\\
4b\equiv r^2-ak^2\pmod{(4\MK,4nk^2)}
}}\legendre{a}{n}.
\end{equation}
The fact that the double infinite sum in~\eqref{avg const} is absolutely convergent 
follows from the proof of Proposition~\ref{weighted avg of l-series}.  
See page~\pageref{weighted avg of l-series}.
However, since the proof of Proposition~\ref{weighted avg of l-series} is similar to 
the proof of~\cite[Lemma 2.2]{DP:2004}, we give only a sketch of the proof, highlighting
the differences.  See~\cite[pp. 193-199]{DP:2004} for more detail.

We also have a description of the constant $\mathfrak{C}_{K,r,1}$ as a product.
However, this requires the introduction of some more notation.
For $b\in G_{\MK}$, let $\Delta^{r,b}:=r^2-4b$, and define the following sets of 
rational primes
\begin{eqnarray}
\mathfrak{Q}^<_{r,b,\MK}
&:=\{\ell>2: \ell|\MK, \ell\dnd r,\text{ and }\ord_\ell(\Delta^{r,b})<\ord_{\ell}(\MK)\},\\
\mathfrak{Q}^\ge_{r,b,\MK}
&:=\{\ell>2: \ell|\MK, \ell\dnd r,\text{ and }\ord_\ell(\Delta^{r,b})\ge\ord_{\ell}(\MK)\}.
\end{eqnarray}
In addition, let
\begin{equation*}
\Gamma_\ell:=\begin{cases}
\legendre{\Delta^{r,b}/\ell^{\ord_\ell(\Delta^{r,b})}}{\ell} &
\text{if }\ord_\ell(\Delta^{r,b})\text{ is even, positive, and finite},\\
0 & \text{otherwise},
\end{cases}
\end{equation*}
and
\begin{equation*}
\mathcal F_2(r,b,\MK):=\begin{cases}
2/3& \begin{array}{l}\text{if }2\dnd r;\end{array}\\
4/3& \begin{array}{l}\text{if }2|r, 4\dnd \MK;\end{array}\\
2-\frac{2}{3\cdot2^{\floor{\ord_2(\MK)/2}}}&
\begin{array}{l}\text{if }r\equiv 2\pmod 4, 
2\le\ord_2(\MK)\le\ord_2(\Delta^{r,b})-2;\end{array}\\
2-\frac{4}{3\cdot 2^{\frac{\ord_2(\MK)-1}{2}}}&
\begin{array}{l}
\text{if }
r\equiv 2\pmod 4, \ord_2(\MK)=\ord_2(\Delta^{r,b})-1,\\
\quad 2|\ord_2(\Delta^{r,b});
\end{array}\\
2-\frac{2}{2^{\ord_2(\MK)/2}}&
\begin{array}{l}
\text{if }r\equiv 2\pmod 4,\ord_2(\MK)=\ord_2(\Delta^{r,b})-1,\\ 
\quad 2\dnd\ord_2(\Delta^{r,b});
\end{array}\\
2-\frac{2}{3\cdot2^{\ord_2(\MK)/2}}&
\begin{array}{l}
\text{if }r\equiv 2\pmod 4,\ord_2(\MK)=\ord_2(\Delta^{r,b}),\\
\quad 2|\ord_2(\Delta^{r,b}),
\frac{\Delta^{r,b}}{2^{\ord_2(\Delta^{r,b})}}\equiv 1\pmod 4;
\end{array}\\
2-\frac{2}{2^{\floor{\ord_2(\MK)/2}}}&
\begin{array}{l}
\text{if } r\equiv 2\pmod 4,
\ord_2(\MK)=\ord_2(\Delta^{r,b}),\\
\quad\left[
2\dnd\ord_2(\Delta^{r,b})\textbf{ OR }
\frac{\Delta^{r,b}}{2^{\ord_2(\Delta^{r,b})}}\equiv 3\pmod 4
\right];
\end{array}\\
2&
\begin{array}{l}
\text{if } r\equiv 2\pmod 4, \ord_2(\MK)>\ord_2(\Delta^{r,b}),\\
\quad 2|\ord_2(\Delta^{r,b}), 
\frac{\Delta^{r,b}}{2^{\ord_2(\Delta^{r,b})}}\equiv 1\pmod 8;
\end{array}\\
2-\frac{4}{3\cdot 2^{\ord_2(\Delta^{r,b})/2}}&
\begin{array}{l}
\text{if } r\equiv 2\pmod 4, \ord_2(\MK)>\ord_2(\Delta^{r,b}),\\
\quad 2|\ord_2(\Delta^{r,b}), 
\frac{\Delta^{r,b}}{2^{\ord_2(\Delta^{r,b})}}\equiv 5\pmod 8;
\end{array}\\
2-\frac{2}{2^{\ord_2(\Delta^{r,b})/2}}&
\begin{array}{l}
\text{if } r\equiv 2\pmod 4, \ord_2(\MK)>\ord_2(\Delta^{r,b}),\\
\quad\left[
2\dnd\ord_2(\Delta^{r,b})\textbf{ OR }
\frac{\Delta^{r,b}}{2^{\ord_2(\Delta^{r,b})}}\equiv 3\pmod 8
\right];
\end{array}\\
\frac{5}{3}&
\begin{array}{l}
\text{if } r\equiv 0\pmod 4, \ord_2(\MK)=2, b\equiv 3\pmod 4;
\end{array}\\
2&
\begin{array}{l}
\text{if } r\equiv 0\pmod 4, 8|\MK, b\equiv 3\pmod 4,\\
\quad\frac{\Delta^{r,b}}{4}\equiv 1\pmod 8;
\end{array}\\
\frac{4}{3}&
\begin{array}{l}
\text{if } r\equiv 0\pmod 4, 8|\MK, b\equiv 3\pmod 4,\\
\quad\frac{\Delta^{r,b}}{4}\equiv 5\pmod 8;
\end{array}\\
1&
\begin{array}{l}
\text{if } r\equiv 0\pmod 4, 4|\MK, b\equiv 1\pmod 4.
\end{array}
\end{cases}
\end{equation*}

Finally, let $\mathcal F(r,b,\MK)$ denote the following finite product over the primes 
dividing $\MK$:
\begin{equation}
\begin{split}
\mathcal F_2(r,b,\MK)
\prod_{\substack{\ell\neq 2\\ \ell|\MK\\ \ell|r}}
\frac{\ell\left(\ell+\legendre{-b}{\ell}\right)}{\ell^2-1}
\prod_{\ell\in\mathfrak{Q}_{r,b,\MK}^\ge}
\left(
\frac{\ell^{\floor{\frac{\ord_\ell(\MK)+1}{2}}}-1}
{\ell^{\floor{\frac{\ord_\ell(\MK)-1}{2}}}(\ell-1)}
+
\frac{\ell^{\ord_\ell(\MK)+2}}{\ell^{3\floor{\frac{\ord_\ell(\MK)+1}{2}}}(\ell^2-1)}
\right)\\
\quad\cdot
\prod_{\ell\in\mathfrak{Q}_{r,b,\MK}^<}
\left(
1+
\frac{\ell\legendre{\Delta^{r,b}}{\ell}+\legendre{\Delta^{r,b}}{\ell}^2
+\displaystyle\frac{\ell\Gamma_\ell+\ell^2\Gamma_\ell^2}{\ell^{\ord_\ell(\Delta^{r,b})/2}}}
{\ell^2-1}
+
\frac{\Gamma_\ell^2\left(\ell^{\floor{\frac{\ord_\ell(\Delta^{r,b})-1}{2}}}-1\right)}
{\ell^{\floor{\frac{\ord_\ell(\Delta^{r,b})-1}{2}}}(\ell-1)}
\right).
\end{split}
\end{equation}

\begin{thm}
As an infinite product over primes, we may write
\begin{equation*}
\mathfrak{C}_{K,r,1}=\left(
\frac{2\NA}{\pi\varphi(\MK)}
\prod_{\substack{\ell\ne 2\\ \ell\dnd\MK\\ \ell\dnd r}}
      \frac{\ell(\ell^2-\ell-1)}{(\ell+1)(\ell-1)^2}
\prod_{\substack{\ell\ne 2\\ \ell\dnd\MK\\ \ell| r}}
      \frac{\ell^2}{\ell^2-1}
\right)
\sum_{b\in G_{\MK}}\mathcal{F}(r,b,\MK).
\end{equation*}
\end{thm}
\begin{rmk}
In the case that $K/\qq$ is an Abelian extension, $K=\mathcal A$ and the constant 
$\mathfrak{C}_{K,r,1}$ agrees with the average order constant of~\cite{FJKP} 
(stated only for odd $r$).
\end{rmk}
\begin{proof}
See Theorem 1.1 and Proposition 2.1 of~\cite{Jam:2005} and compare with~\eqref{avg const} 
on page~\pageref{avg const}.
\end{proof}

\section{Intermediate results}\label{intermediate results}

\subsection{Counting curves}

For a fixed prime ideal $\p$ and a fixed elliptic curve $E$ defined over the finite field 
$\ints_K/\p$, we will need to count the number of models in the box $\ecbox$ which are isomorphic 
to $E$ modulo $\p$.  Now, if we assume that $\p\dnd 6$, then any elliptic curve defined over 
$\ints_K/\p$ may be realized using a model of the form 
\begin{equation}
E_{a,b}: Y^2=X^3+aX+b\text{ with }a,b\in\ints_K/\p.
\end{equation}
For an elliptic curve $E$ over $K$ and a prime $\p$ of good reduction, 
let $E^\p$ denote the reduction of $E$ modulo $\p$.
In order to carry out the proof of Theorem~\ref{avg LT for K}, we will need an estimate on the 
size of
\begin{equation}
\mathscr{B}(E_{a,b},\p):=\{E\in\mathscr{B}: E^\p\isom E_{a,b}\}.
\end{equation}

\begin{rmk}
It is important to note that here we are counting equations (or models) in 
$\mathscr{B}$ whose reduction modulo $\p$ are in the same isomorphism 
class as $E_{a,b}$.
\end{rmk}

\begin{lma}\label{count isom reductions}
Recall that $\mathcal B=\{\gamma_j\}_{j=1}^{\NK}$ is our fixed integral basis for $\ints_K$.  
Let $\p$ be a degree $1$ prime of $K$ such that $\p\dnd 6\prod_{j=1}^{\NK}\gamma_j$.
Let $p$ denote the 
unique rational prime lying below $\p$, and let $E_{a,b}$ be the fixed 
elliptic curve defined over $\ints_K/\p\isom\gf_p$ via the equation 
$E_{a,b}: Y^2=X^3+aX+b$.  Then
\begin{equation*}
\begin{split}
\#\ecbox(E_{a,b},\p)=\frac{p-1}{p^2\#\mathrm{Aut}_p(E_{a,b})}\mathscr{V}(\ecbox)
&+O\left(
\frac{\mathscr{V}(\ecbox)}{p\mathscr{V}_{\mathrm{min}}(\ecbox)}
+\left(\mathscr{V}_1(\ecbox)+\mathscr{V}_2(\ecbox)\right)p^{\NK-3/2}(\log p)^{\NK}
\right)\\
&+O\left(p^{2\NK-3/2}(\log p)^{2\NK}\right),
\end{split}
\end{equation*}
where
\begin{equation*}
\#\mathrm{Aut}_p(E_{a,b})=\begin{cases}
2&\text{if } ab\ne 0,\\
(4,p-1)&\text{if } a\ne 0\text{ and }b=0,\\
(6,p-1)&\text{if } a=0\text{ and }b\ne 0.
\end{cases}
\end{equation*}
\end{lma}

We delay the proof of Lemma~\ref{count isom reductions} until 
\textsection\ref{reduction count proof sect}.  
In addition, if $\p$ and $\p'$ do not lie above the same rational prime, we will also
need an estimate on the size of
\begin{equation}
\ecbox(E_{a,b},\p; E_{a',b'},\p'):=\{E\in\ecbox: E^\p\isom E_{a,b}\text{ and } E^{\p'}\isom E_{a',b'}\}
\end{equation}
in order to carry out the proof of Theorem~\ref{variance thm}.

\begin{lma}\label{count isom reduction pairs}
With the same notation and assumptions as in Lemma~\ref{count isom reductions}, assume 
further that $\p'$ is a prime of $K$ satisfying the same conditions as $\p$ except that 
$\p'$ lies over the rational prime $p'$ and $p'\ne p$.  That is, $\p$ and $\p'$ do not lie over 
the same prime.  Also, let $E_{a',b'}$ be the fixed elliptic curve defined over 
$\ints_K/\p'\isom\gf_{p'}$ via the equation $E_{a',b'}: Y^2=X^3+a'X+b'$.  Then
\begin{equation*}
\begin{split}
\#\ecbox(E_{a,b},\p;E_{a',b'},\p')&=
\frac{(p-1)(p'-1)}{(pp')^2\#\mathrm{Aut}_p(E_{a,b})\#\mathrm{Aut}_{p'}(E_{a',b'})}\mathscr{V}(\ecbox)
+O\left(
\frac{\mathscr{V}(\ecbox)}{pp'\mathscr{V}_{\mathrm{min}}(\ecbox)}
\right)\\
&+O\left(
(pp')^{2\NK-3/2}(\log pp')^{2\NK}
+\left(\mathscr{V}_1(\ecbox)+\mathscr{V}_2(\ecbox)
\right)(pp')^{\NK-3/2})(\log pp')^{\NK}
\right).
\end{split}
\end{equation*}
\end{lma}

Because the proof of Lemma~\ref{count isom reduction pairs} is similar to the proof of 
Lemma~\ref{count isom reductions} except that it is more tedious, we omit it.

\subsection{A weighted average of special values of Dirichlet $L$-functions}\label{L-fct sect}
Recall that for a Dirichlet character $\chi$ and for $\Re(s)$ sufficiently large, 
the Dirichlet $L$-function associated to $\chi$ is defined by
\begin{equation*}
L(s,\chi):=\sum_{n\ge 1}\frac{\chi(n)}{n^s}.
\end{equation*}
Also, recall that when $\chi$ is not trivial, this series converges at $s=1$.
For an integer $d$, we let $\chi_d$ denote the Kronecker symbol $\legendre{d}{\cdot}$.

Let $B(r):=\max\{5,r^2/4\}$, and
let $\mathscr P_r$ denote the set of prime ideals $\p$ satisfying:\label{prime sets}
\begin{itemize}\label{prime set}
\item $B(r)<\N\p$,
\item $\p$ lies over some rational prime $p$ which splits completely in $K$,
\item $\p\dnd\prod_{j=1}^{\NK}\gamma_j$,
\item $\p$ does not ramify in $K(\zeta_{\MK})$.
\end{itemize}
We also define the corresponding set of ``downstairs" primes.
That is, we let $\mathcal P_r$ denote the set of rational primes $p$ lying below some prime 
$\p\in\mathscr P_r$.  
In addition, we will require the ``truncated" sets
$\mathscr{P}_r(x):=\left\{\p\in \mathscr{P}_r: \N\p\in (1,x]\right\}$ and
$\mathcal{P}_r(x):=\mathcal{P}_r\cap (1,x]$.

For a prime $p$ and a positive integer $k$, let $d_k(p):=(r^2-4p)/k^2$ 
if $k^2$ divides $r^2-4p$.  Finally, define the set
\begin{equation}
\mathcal S_k(x;r):=\{p\in\mathcal P_r(x): k^2|(4p-r^2), d_k(p)\equiv 0,1\pmod{4}\}.
\end{equation}

\begin{prop}\label{weighted avg of l-series}
Let
\begin{equation*}
A_1(x;r):=\NK\sum_{k\le 2\sqrt x}\frac{1}{k}
\sum_{p\in\mathcal{S}_k(x;r)}L(1,\chi_{d_k(p)})\log p.
\end{equation*}
Then the double infinite sum defining $\mathfrak{C}_{K,r,1}$ in~\eqref{avg const} 
is absolutely convergent; and for any $\eta>0$,
\begin{equation*}
A_1(x;r)=\frac{\pi}{2}\mathfrak{C}_{K,r,1}x+O\left(\frac{x}{(\log x)^\eta}\right).
\end{equation*}
\end{prop}
A sketch of the proof of Proposition~\ref{weighted avg of l-series} is 
given in Section~\ref{compute l-series avg}.  For the omitted details, we refer 
the reader to~\cite[pp. 193-199]{DP:2004} since it is similar.

\section{The average order}

We now use the results of Section~\ref{intermediate results} to compute 
the average order of $\pi_E^{r,1}(x)$.  That is, we prove 
Theorem~\ref{avg LT for K}.
We compute the average order by first converting it into a weighted sum of 
class numbers.  
Given a (not necessarily fundamental) discriminant $D<0$, 
we define the \textit{Hurwitz-Kronecker class number} of discriminant $D$ by
\begin{equation}\label{Hurwitz defn}
H(D):=2\sum_{\substack{k^2|D\\ \frac{D}{k^2}\equiv 0,1\pmod 4}}
\frac{h(D/k^2)}{w(D/k^2)},
\end{equation}
where $h(d)$ denotes the class number of the unique imaginary quadratic order of 
discriminant $d$ and $w(d)$ denotes the order of its unit group.

The following result of Deuring is the key to counting elliptic curves 
over a finite field.  See~\cite{Deu:1941} or~\cite[p.~654]{Len:1987}.
\begin{thm}[Deuring]\label{Deuring's thm}
Let $p$ be prime greater than $3$ and $r$ an integer satisfying $r^2-4p<0$.  Then 
\begin{equation*}
\sum_{\substack{\tilde E/\gf_p\\ \#\tilde E(\gf_p)=p+1-r}}\frac{1}{\#\Aut(\tilde E)}=\frac{1}{2}H(r^2-4p),
\end{equation*}
where the sum on the left is over the $\gf_p$-isomorphism classes of elliptic curves 
having exactly $p+1-r$ points and $\#\Aut(\tilde E)$ denotes the size of the automorphism 
group of any representative of the class $\tilde E$.
\end{thm}
\begin{rmk}
It is important to note that our definition of $H(D)$ is defined as a weighted sum of ordinary 
class numbers $h(d)$.  Thus, our statement of Deuring's Theorem looks more 
like that given in~\cite[p.~654]{Len:1987} as opposed to that given 
in~\cite[Theorem 4.6]{Sch:1987}.
However, our definition is exactly twice as big as the definition used in~\cite{Len:1987} and the 
statement of Deuring's Theorem is adjusted accordingly.
\end{rmk}

\begin{prop}\label{1st avg order prop}
If $\mathcal P_r(x)$ is the set of primes defined in~\textsection\ref{L-fct sect}, then
\begin{equation*}
\frac{1}{\#\ecbox}\sum_{E\in\ecbox}\pi_E^{r,1}(x)=
\frac{\NK}{2}\sum_{p\in\mathcal P_r(x)}\frac{H(r^2-4p)}{p}
+O(\mathcal{E}_0(x;\mathscr B)),
\end{equation*}
where
\begin{equation*}
\mathcal E_0(x;\ecbox):=1
+\frac{\sqrt x/\log x}{\mathscr{V}_{\mathrm{min}}(\ecbox)}
+\left(\frac{1}{\mathscr{V}_1(\ecbox)}+\frac{1}{\mathscr{V}_2(\ecbox)}\right)(x\log x)^{\NK}
+\frac{(x\log x)^{2\NK}}{\mathscr{V}(\ecbox)}.
\end{equation*}
\end{prop}
\begin{proof}
Since $\mathscr{P}_r$ contains all but finitely many degree $1$ primes of $K$, we have
\begin{equation}
\begin{split}\label{interchange sums}
\frac{1}{\#\ecbox}\sum_{E\in\ecbox}\pi_E^{r,1}(x)
&=\frac{1}{\#\ecbox}\sum_{E\in\ecbox}
\sum_{\substack{\p\in\mathscr{P}_r(x)\\ a_\p(E)=r}}1 +O(1)
=\frac{1}{\#\ecbox}\sum_{\substack{\p\in\mathscr P_r(x)}}
\sum_{\substack{E\in\ecbox\\ a_\p(E)=r}}1 +O(1)\\
&=\frac{1}{\#\ecbox}\sum_{\substack{\p\in\mathscr P_r(x)}}
\sum_{\substack{\tilde E/\gf_p\\ \#\tilde{E}(\gf_p)=p+1-r}}\#\ecbox(\tilde E, \p)+
O\left(\frac{1}{\#\ecbox}\sum_{\p\in\mathscr P_r(x)}
\sum_{\substack{E\in\ecbox\\ E_\p\text{ sing.}}}1\right),
\end{split}
\end{equation}
where $p=\N\p$ and the inner sum of the final big-$O$ term is over the $E\in\ecbox$ whose 
reductions are singular modulo $\p$.  
Using character sums in a manner similar to the proof of 
Lemma~\ref{count isom reductions} (see \textsection\ref{reduction count proof sect}), we obtain a 
bound on the big-$O$ term which is smaller than $\mathcal E_0(x;\ecbox)$.  Therefore, we will
concentrate on the main term of~\eqref{interchange sums}.

By Lemma~\ref{count isom reductions}, we see that 
$\#\ecbox(\tilde E,\p)$ depends only on the rational prime $p$ lying below $\p$.  
Since there are exactly $\NK$ degree $1$ primes $\p$ lying above the same rational prime $p$,
we obtain
\begin{equation*}
\frac{1}{\#\ecbox}\sum_{\substack{\p\in\mathscr P_r(x)}}
\sum_{\substack{\tilde E/\gf_p\\ \#\tilde{E}(\gf_p)=p+1-r}}\#\ecbox(\tilde E, \p)
=\frac{\NK}{\#\ecbox}\sum_{\substack{p\in\mathcal P_r(x)}}
\sum_{\substack{\tilde E/\gf_p\\ \#\tilde{E}(\gf_p)=p+1-r}}\#\ecbox(\tilde E, \p).
\end{equation*}
Finally, using Lemma~\ref{count isom reductions} to estimate $\#\ecbox(\tilde E,\p)$ and 
Theorem~\ref{Deuring's thm} to count $\gf_p$-isomorphism classes of elliptic curves with 
exactly $p+1-r$ points,
we obtain
\begin{equation}
\frac{1}{\#\ecbox}\sum_{\substack{\p\in\mathscr P_r(x)}}
\sum_{\substack{\tilde E/\gf_p\\ \#\tilde{E}(\gf_p)=p+1-r}}\#\ecbox(\tilde E, \p)
=\frac{\NK}{2}\sum_{p\in\mathcal P_r(x)}\frac{H(r^2-4p)}{p}
+\sum_{p\in\mathcal P_r(x)}\mathcal E_p(x;\ecbox),
\end{equation}
where
\begin{equation*}
\begin{split}
\mathcal E_p(x;\ecbox):=\frac{H(r^2-4p)}{p^2}+
\frac{H(r^2-4p)}{p\mathscr{V}_{\mathrm{min}}(\ecbox)}
&+p^{\NK-3/2}(\log p)^{\NK}H(r^2-4p)
\left(\frac{1}{\mathscr{V}_1(\ecbox)}+\frac{1}{\mathscr V_2(\ecbox)}\right)\\
&+\frac{p^{2\NK-3/2}(\log p)^{2\NK}H(r^2-4p)}{\mathscr{V}(\ecbox)}.
\end{split}
\end{equation*}
In~\cite[p.178]{DP:1999}, we find the bound
\begin{equation*}
\sum_{p\le x}H(r^2-4p)\ll x^{3/2}.
\end{equation*}
Using this bound, together with partial summation and standard estimates, we obtain 
\begin{equation}
\sum_{p\in\mathcal P_r(x)}\mathcal E_p(x;\ecbox)\ll
1
+\frac{\sqrt x/\log x}{\mathscr{V}_{\mathrm{min}}(\ecbox)}
+\left(\frac{1}{\mathscr{V}_1(\ecbox)}+\frac{1}{\mathscr{V}_2(\ecbox)}\right)(x\log x)^{\NK}
+\frac{(x\log x)^{2\NK}}{\mathscr{V}(\ecbox)}.
\end{equation}
This completes the proof of the proposition.
\end{proof}

\begin{prop}\label{2nd avg order prop}
For every $\eta>0$,
\begin{equation*}
\frac{\NK}{2}\sum_{p\in\mathcal P_r(x)}\frac{H(r^2-4p)}{p}
=\mathfrak{C}_{K,r,1}\pi_{1/2}(x)+O\left(\frac{\sqrt x}{(\log x)^{1+\eta}}\right).
\end{equation*}
\end{prop}
\begin{proof}
From the definition of the Kronecker-Hurwitz class number~\eqref{Hurwitz defn} 
and Dirichlet's class number formula~\cite[p. 513]{IK:2004}, we have
\begin{equation}\label{apply class number formula}
\begin{split}
\frac{\NK}{2}\sum_{p\in\mathcal P_r(x)}\frac{H(r^2-4p)}{p}
&=\frac{\NK}{2}\sum_{p\in\mathcal P_r(x)}
\sum_{\substack{k^2|(r^2-4p)\\ d_k(p)\equiv 0,1\pmod 4}}
\frac{\sqrt{4p-r^2}}{\pi kp}L(1,\chi_{d_k(p)})\\
&=\frac{\NK}{\pi}\sum_{k\le 2\sqrt x}\frac{1}{k}\sum_{p\in\mathcal S_k(x;r)}
\frac{L(1,\chi_{d_k(p)})}{\sqrt p}
+O\left(
\sum_{p\le x}\frac{\log p}{p^{3/2}}\sum_{k^2|(r^2-4p)}1
\right),
\end{split}
\end{equation}
where we have used the fact that $\sqrt{4p-r^2}=2\sqrt p+O(p^{-1/2})$ together 
with the bound $L(1,\chi_d)\ll\log d$.  See~\cite[p. 120]{IK:2004} for example.
Since $\sum_{k^2|(4p-r^2)}1\ll p^\delta$ for any $\delta>0$, the big-$O$ term is bounded.

Partial summation applied to the inner sum of the main term yields
\begin{equation}
\begin{split}
\sum_{p\in\mathcal S_k(x;r)}\frac{L(1,\chi_{d_k(p)})}{\sqrt p}
=&\frac{1}{\sqrt x\log x}\sum_{p\in\mathcal S_k(x;r)}L(1,\chi_{d_k(p)})\log p\\
&+\int_{B(r)}^x
\frac{\sum_{p\in\mathcal S_k(t;r)}L(1,\chi_{d_k(p)})\log p}{2t^{3/2}\log t+t^{3/2}(\log t)^2}\mathrm dt.
\end{split}
\end{equation}
Substituting this back into~\eqref{apply class number formula}, applying 
Proposition~\ref{weighted avg of l-series}, and using~\eqref{pi1/2}, we have
\begin{equation*}
\begin{split}
\frac{\NK}{2}\sum_{p\in\mathcal P_r(x)}\frac{H(r^2-4p)}{p}
&=\frac{1}{\pi\sqrt x\log x}A_1(x;r)
-\int_{B(r)}^x\frac{A_1(t;r)}{2t^{3/2}\log t+t^{3/2}(\log t)^2}\mathrm dt\\
&=\frac{\mathfrak{C}_{K,r,1}}{2}\left[
\frac{\sqrt x}{\log x}
+\int_2^x\frac{\mathrm dt}{2\sqrt t\log t}
+\int_2^x\frac{\mathrm dt}{\sqrt t(\log t)^2}
\right]
+O\left(
\frac{\sqrt x}{(\log x)^{1+\eta}}
\right)\\
&=\mathfrak{C}_{K,r,1}\pi_{1/2}(x)+O\left(\frac{\sqrt x}{(\log x)^{1+\eta}}\right).
\end{split}
\end{equation*}
\end{proof}
Theorem~\ref{avg LT for K} now follows by combining the results of
Propositions~\ref{1st avg order prop}  and~\ref{2nd avg order prop}.

\section{The variance}

In this section, we bound the variance of $\pi_E^{r,1}(x)$.  That is, we 
give the proof of Theorem~\ref{variance thm} using the results of 
Section~\ref{intermediate results}.
\begin{proof}[Proof of Theorem~\ref{variance thm}]
Expanding the square and applying Theorem~\ref{avg LT for K}, we find that
\begin{equation}\label{use variance formula}
\begin{split}
\frac{1}{\#\ecbox}\sum_{E\in\ecbox}\left|\pi_{E}^{r,1}(x)-\mathfrak{C}_{K,r,1}\pi_{1/2}(x)\right|^2
=\frac{1}{\#\ecbox}\sum_{E\in\ecbox}\left[\pi_E^{r,1}(x)\right]^2
   &-\left[\mathfrak C_{K,r,1}\pi_{1/2}(x)\right]^2 \\
   &+O\left(\frac{\mathcal E(x;\ecbox)\sqrt x}{\log x}\right).
 \end{split}
\end{equation}
In the sum on the right-hand side of~\eqref{use variance formula}, 
we again expand the square and group terms according to pairs of 
prime ideals of equal norm and pairs of unequal norm.  We obtain
\begin{equation}\label{split up the primes}
\frac{1}{\#\ecbox}\sum_{E\in\ecbox}\left[\pi_E^{r,1}(x)\right]^2
 =\frac{1}{\#\ecbox}\sum_{E\in\ecbox}
       \left[\sum_{\substack{\N\p,\N\p'\le x\\ \N\p=\N\p'\\ a_\p(E)=a_{\p'}(E)=r\\ \deg\p=\deg\p'=1}}1
       +\sum_{\substack{\N\p,\N\p'\le x\\ \N\p\neq\N\p'\\ a_\p(E)=a_{\p'}(E)=r\\ \deg\p=\deg\p'=1}}1\right].
\end{equation}
We bound the sum over the prime pairs of equal norm by observing that
\begin{equation*}
\frac{1}{\#\ecbox}\sum_{E\in\ecbox}
    \sum_{\substack{\N\p,\N\p'\le x\\ \N\p=\N\p'\\ a_\p(E)=a_{\p'}(E)=r\\ \deg\p=\deg\p'=1}}1
\le\frac{1}{\#\ecbox}\sum_{E\in\ecbox}
    \sum_{\substack{\N\p\le x\\ a_\p(E)=r\\ \deg\p=1}}\NK
=\frac{\NK}{\#\ecbox}\sum_{E\in\ecbox}\pi_E^{r,1}(x).
\end{equation*}
Thus, applying Theorem~\ref{avg LT for K}, we have
\begin{equation}\label{bound primes of equal norm}
\frac{1}{\#\ecbox}\sum_{E\in\ecbox}
    \sum_{\substack{\N\p,\N\p'\le x\\ \N\p=\N\p'\\ a_\p(E)=a_{\p'}(E)=r\\ \deg\p=\deg\p'=1}}1
\ll\frac{1}{\#\ecbox}\sum_{E\in\ecbox}\pi_E^{r,1}(x)
\ll\frac{\sqrt x}{\log x}+\mathcal E(x;\ecbox).
\end{equation}

For the primes of unequal norm, we argue as in the proof of Proposition~\ref{1st avg order prop}
and write
\begin{equation}\label{primes of unequal norm}
\begin{split}
\frac{1}{\#\ecbox}\sum_{E\in\ecbox}
   \sum_{\substack{\N\p,\N\p'\le x\\ \N\p\neq\N\p'\\ a_\p(E)=a_{\p'}(E)=r\\ \deg\p=\deg\p'=1}}1
=&\sum_{\substack{\p,\p'\in\mathscr P_r(x)\\ \N\p\neq\N\p'}}
    \sum_{\substack{\tilde E/\gf_p, \tilde E'/\gf_{p'}\\ \#E(\gf_p)=p+1-r\\ \#E'(\gf_{p'})=p'+1-r}}
    \#\ecbox(E,\p; E',\p')\\
   &+O\left(
   	\frac{1}{\#\ecbox}\sum_{\substack{\p,\p'\in\mathscr P_r(x)\\ \N\p\neq\N\p'}}
	\sum_{\substack{E\in\ecbox\\ E_\p\text{ or } E_{\p'}\text{ sing.}}}
	  \right).
\end{split}
\end{equation}
As in the proof of Proposition~\ref{1st avg order prop}, the contribution of the big-$O$ term is 
negligible.  Applying Lemma~\ref{count isom reduction pairs} and 
Theorem~\ref{Deuring's thm} to estimate the main term 
of~\eqref{primes of unequal norm}, we see that the contribution from the pairs of primes of 
unequal norm is equal to
\begin{equation}\label{convert unequal sum to class number sum}
\frac{1}{\#\ecbox}\sum_{E\in\ecbox}
   \sum_{\substack{\N\p,\N\p'\le x\\ \N\p\neq\N\p'\\ a_\p(E)=a_{\p'}(E)=r\\ \deg\p=\deg\p'=1}}1
=\frac{\NK^2}{4}\sum_{\substack{p,p'\in\mathcal P_r(x)\\ p\neq p'}}
  \frac{H(r^2-4p)H(r^2-4p')}{pp'}
  +O\left(
  	\mathcal E'(x;\ecbox)
       \right),
\end{equation}
where
\begin{equation*}
\mathcal E'(x;\ecbox):=\frac{\sqrt x\log\log x}{\log x}
   +\frac{x/(\log x)^2}{\mathscr{V}_{\mathrm{min}}(\ecbox)}
   +\left(
   	\frac{1}{\mathscr V_1(\ecbox)}+\frac{1}{\mathscr{V}_2(\ecbox)}
      \right)(x\log x)^{2\NK}
    +\frac{(x\log x)^{4\NK}}{\mathscr{V}(\ecbox)}.
\end{equation*}
By Proposition~\ref{2nd avg order prop}, the double sum over primes 
in~\eqref{convert unequal sum to class number sum} is equal to
\begin{equation}\label{apply avg order prop to compute double Hurwitz sum}
\left[\frac{\NK}{2}\sum_{p\in\mathcal P_r(x)}\frac{H(r^2-4p)}{p}\right]^2
-\frac{\NK^2}{4}\sum_{p\in\mathcal P_r(x)}\frac{H(r^2-4p)^2}{p^2}
=\left[\mathfrak C_{K,r,1}\pi_{1/2}(x)\right]^2
  +O\left(
  	\frac{x}{(\log x)^{2+\eta}}
        \right).
\end{equation}
Combining~\eqref{split up the primes},
~\eqref{bound primes of equal norm},
~\eqref{convert unequal sum to class number sum},
and~\eqref{apply avg order prop to compute double Hurwitz sum}, we have
\begin{equation*}
\frac{1}{\#\ecbox}\sum_{E\in\ecbox}\left[\pi_E^{r,1}(x)\right]^2
=\left[\mathfrak C_{K,r,1}\pi_{1/2}(x)\right]^2
	+O\left(
		\frac{x}{(\log x)^{2+\eta}}+\mathcal E'(x;\ecbox)
	     \right).
\end{equation*}
Substituting this 
into~\eqref{use variance formula}, we obtain the desired result.
\end{proof}

\section{Counting elliptic curves whose reductions are isomorphic over $\gf_p$}\label{reduction count proof sect}

In this section, we will prove the estimates of Lemma~\ref{count isom reductions} by extending 
standard character sum techniques as in~\cite{FM:1996} or~\cite{JY:2006}.  
The main difference is that we have to extend 
the domain of our characters to $\ints_K$.  Since we are representing $\ints_K$ as an 
$\NK$-dimensional $\zz$-module, it is necessary to adapt their proof to higher dimensions.

\begin{proof}[Proof of Lemma~\ref{count isom reductions}]
We begin by recalling that $E_{a,b}$ and $E_{a',b'}$ are isomorphic over $\gf_p$ if and only if 
there exists a $u\in\gf_p^*$ so that $a=u^4a'$ and $b=u^6b'$.  Now, note that 
$\#\mathrm{Aut}_p(E_{a,b})=\#\{u\in\gf_p^*: a=au^4\text{ and } b=bu^6\}$.  Therefore, we may write
\begin{equation}
\#\ecbox(E_{a,b},\p)=\frac{1}{\#\mathrm{Aut}_p(E_{a,b})}\sum_{u\in(\ints_K/\p)^*}
\sum_{\substack{\alpha\in\mathcal B(\vect a_1,\vect b_1)\\ \p|(\alpha-au^4)}}
\sum_{\substack{\beta\in\mathcal B(\vect a_2,\vect b_2)\\ \p|(\beta-bu^6)}}1.
\end{equation}
Refer back to equation~\eqref{int box} on page~\ref{int box} for the definition of 
$\mathcal B(\vect a,\vect b)$.  To estimate this sum, we write
\begin{equation}
\#\ecbox(E_{a,b},\p)=\frac{1}{\#\mathrm{Aut}_p(E_{a,b})}
\sum_{u\in(\ints_K/\p)^*}
\sum_{\substack{\alpha\in\mathcal B(\vect a_1,\vect b_1)\\ \beta\in\mathcal{B}(\vect a_2,\vect b_2)}}
\frac{1}{p^2}\sum_{(\psi,\psi')}\psi(\alpha-au^4)\psi'(\beta-bu^6),
\end{equation}
where the innermost sum is over all pairs $(\psi,\psi')$ of additive characters on 
$\ints_K/\p\isom\gf_p$.  The main term results from when $\psi=\psi'=\psi_0$, the trivial 
character, which contributes
$\frac{p-1}{p^2\#\mathrm{Aut}_p(E_{a,b})}\mathscr{V}(\ecbox)
+O\left(\frac{\mathscr{V}(\ecbox)}{p\mathscr{V}_{\mathrm{min}}(\ecbox)}\right)$.
The remaining terms are bounded by
\begin{equation}
\frac{1}{p^2\#\mathrm{Aut}_p(E_{a,b})}\sum_{(\psi,\psi')\ne (\psi_0,\psi_0)}
\left|
\sum_{u\in(\ints_K/\p)^*}\overline\psi(au^4)\overline{\psi'}(bu^6)
\right|
\left|
\sum_{\alpha\in\mathcal B(a_1,b_1)}\psi(\alpha)
\right|
\left|
\sum_{\beta\in\mathcal B(a_2,b_2)}\psi(\beta)
\right|.
\end{equation}

Since in the line above at least one of $\psi$ and $\psi'$ is not trivial, we will assume for the 
moment that it is $\psi$.  Using well known facts about additive characters modulo $p$, 
we may write
\begin{equation*}
\overline\psi(au^4)\overline{\psi'}(bu^6)
=\overline\psi(au^4+mbu^6)
\end{equation*}
for some $m\in\gf_p$.
We think of the expression $au^4+mbu^6$ as a polynomial in $u$ over $\ints_K/\p\isom\gf_p$ 
of degree either $4$ or $6$.  Thus, since $\p\dnd 6$, we may apply 
Weil's Theorem~\cite[p. 223]{LN:1997}, which yields
\begin{equation*}
\left|
\sum_{u\in(\ints_K/\p)^*}\overline\psi(au^4+mbu^6)
\right|\ll\sqrt p.
\end{equation*}

We now estimate $\sum_{\psi}\left|\sum_{\alpha\in\mathcal B(a_1,b_1)}\psi(\alpha)\right|$.
If $\psi=\psi_0$, then 
$\left|\sum_{\alpha\in\mathcal B(a_1,b_1)}\psi(\alpha)\right|\ll\mathscr{V}_1(\ecbox)$ and 
it remains to bound 
$\sum_{\psi\ne\psi_0}\left|\sum_{\alpha\in\mathcal B(a_1,b_1)}\psi(\alpha)\right|$.
We now write each $\alpha$ in terms of our fixed basis 
$\mathcal B=\{\gamma_1,\dots,\gamma_{\NK}\}$ as $\alpha=\sum_{j=1}^{\NK}c_j\gamma_j$.
Since we assumed that $\p\dnd\prod_{j=1}^{\NK}\gamma_j$, each $\gamma_j$ is nonzero 
modulo $\p$.  Thus,
\begin{equation*}
\begin{split}
\sum_{\psi\ne\psi_0}\left|\sum_{\alpha\in\mathcal B(a_1,b_1)}\psi(\alpha)\right|
&=\sum_{\psi\ne\psi_0}\left|\sum_{\alpha\in\mathcal B(a_1,b_1)}
\prod_{j=1}^{\NK}\psi(c_j\gamma_j)\right|\\
&=
\sum_{\psi\ne\psi_0}
\prod_{j=1}^{\NK}\left|
\sum_{c_j=\ceil{a_{1,j}-b_{1,j}}}^{\floor{a_{1,j}+b_{1,j}}}
\psi(c_j\gamma_j)\right|\\
&\le
\prod_{j=1}^{\NK}
\sum_{\psi\ne\psi_0}
\left|
\sum_{c_j=\ceil{a_{1,j}-b_{1,j}}}^{\floor{a_{1,j}+b_{1,j}}}
\psi(c_j\gamma_j)\right|\\
&=\prod_{j=1}^{\NK}
\sum_{\psi\ne\psi_0}
\left|
\sum_{c_j=\ceil{a_{1,j}-b_{1,j}}}^{\floor{a_{1,j}+b_{1,j}}}
\psi(c_j)\right|\\
&\ll p^{\NK}(\log p)^{\NK}.
\end{split}
\end{equation*}

The same line of reasoning suffices to estimate 
$\sum_{\psi}\left|\sum_{\beta\in\mathcal B(a_2,b_2)}\psi(\beta)\right|$.
The result now follows by appropriately combining all of these estimates.
\end{proof}

\section{Averaging special values of Dirichlet $L$-functions}\label{compute l-series avg}
In this section, we sketch the proof of Proposition~\ref{weighted avg of l-series}.
That is, we show how to compute
\begin{equation}
A_1(x;r)=\NK\sum_{k\le 2\sqrt x}\frac{1}{k}
\sum_{p\in\mathcal{S}_k(x;r)}L(1,\chi_{d_k(p)})\log p.
\end{equation}
Similar averages of the special values of Dirichlet $L$-functions arise in 
previous work on the average Lang-Trotter problem.
The general strategy has been to reorder the summation so that one arrives at a sum 
that can be easily estimated using the Prime Number Theorem for primes in arithmetic 
progressions.  For the case of degree $1$ primes of a number field $K$, one needs 
to estimate sums of the form
\begin{equation*}
\sum_{\substack{p\le x\\ p\text{ splits comp. in } K\\ p\equiv a\pmod q}}\log p
\end{equation*}
for essentially every possible value of $a$ and $q$.  

When $K/\qq$ is Abelian, as in~\cite{FJKP}, the condition that $p$ splits completely 
in $K$ is determined by congruence conditions.  That is, there exists an integer 
$\MK$ and a subgroup $G_{\MK}\subseteq(\zz/\MK\zz)^*$ so that $p$ splits 
completely in $K$ if and only if $p\equiv b\pmod{\MK}$ for some $b\in G_{\MK}$.
One may check that this definition of $G_{\MK}$ agrees with the one given 
on page~\pageref{def GmK} in the case that $K=\mathcal A$ is an Abelian 
extension of $\qq$.
Thus, if $K/\qq$ is Abelian, one may rewrite the above sum as
\begin{equation*}
\sum_{\substack{p\le x\\ p\text{ splits comp. in } K\\ p\equiv a\pmod q}}\log p
=
\sum_{b\in G_{\MK}}\sum_{\substack{p\le x\\ p\equiv b\pmod{\MK}\\ p\equiv a\pmod q}}\log p.
\end{equation*}
The inner sum can then be estimated by the Prime Number Theorem for primes in 
arithmetic progressions when $a\equiv b\pmod{(q,\MK)}$.  Otherwise, the sum is empty.

When $K/\qq$ is a non-Abelian Galois extension, it is not likely that one will be able to 
write down a list of congruence conditions that determine exactly when a rational prime 
will split completely in $K$.  Essentially, the remedy is apply the Chebotar\"ev Density 
Theorem to the appropriate Galois extension of $K$.
In order to transform our sum into a form appropriate for application of the 
Chebotar\"ev Density Theorem, we make the following simple observation.
For each rational prime $p$ splitting completely in $K$, there are 
exactly $\NK$ primes $\p$ of $K$ lying above $p$, all satisfying $\N\p=p$.  Therefore,
\begin{equation}\label{back upstairs}
\sum_{\substack{p\le x\\ p\text{ splits comp. in } K\\ p\equiv a\pmod q}}\log p
=\frac{1}{\NK}\sum_{\substack{\N\p\le x\\ \deg\p=1\\ \N\p\equiv a\pmod q}}\log\N\p
+O(1).
\end{equation}
The sum on the right hand side is now in appropriate form to be estimated by the 
Chebotar\"ev Density Theorem.  We now explain this application.

For each positive integer $q$, let $\zeta_q$ be a primitive $q$-th root of unity and 
 let $G_q$ denote the image of the natural map
\begin{equation}\label{natural map}
\begin{diagram}
\node{\Gal(K(\zeta_q)/K)}\arrow{e,J}
\node{\Gal(\qq(\zeta_q)/\qq)}\arrow{e,t}{\sim}
\node{(\zz/q\zz)^*.}
\end{diagram}
\end{equation}
Thus, for each positive integer $q$, we have a canonical identification of 
$\Gal(K(\zeta_q)/K)$ with a certain subgroup of $(\zz/q\zz)^*$, which we denote by $G_q$.
For the case that $q=\MK$, it is easy to check that this definition of $G_{\MK}$ agrees 
with the one given on page~\pageref{def GmK}.  Indeed,
$G_{\MK}\isom\Gal(K(\zeta_{\MK})/K)\isom\Gal(\qq(\zeta_{\MK})/\mathcal A)$.

For each prime ideal $\p$ of $K$ not ramifying in $K(\zeta_q)$, it is easy to check that 
the Frobenius automorphism at $\p$ is determined by the residue of $\N\p$ modulo $q$.
Thus, provided that $a$ is in the image of the natural map~\eqref{natural map}, it 
follows from Chebotar\"ev Density Theorem that 
\begin{equation*}
\sum_{\substack{\N\p\le x\\ \N\p\equiv a\pmod q}}\log\N\p
\sim\frac{1}{\varphi_K(q)}x,
\end{equation*}
where $\varphi_K(q):=\#G_q$.
Otherwise, the sum is empty.
Since the contribution from the primes $\p$ of degree greater than or equal to $2$ is 
$O(\sqrt x)$ (with an implied constant depending on $K$, but not on $q$), 
it also follows that
\begin{equation*}
\sum_{\substack{\N\p\le x\\ \deg\p=1\\ \N\p\equiv a\pmod q}}\log\N\p
\sim\frac{1}{\varphi_K(q)}x.
\end{equation*}
Applying this in~\eqref{back upstairs} yields the asymptotic identity
\begin{equation*}
\sum_{\substack{p\le x\\ p\text{ splits comp. in } K\\ p\equiv a\pmod q}}\log p
=\frac{1}{\NK}\sum_{\substack{\N\p\le x\\ \deg\p=1\\ \N\p\equiv a\pmod q}}\log\N\p
\sim\frac{1}{\NK\varphi_K(q)}x.
\end{equation*}
Essentially, by replacing our sum over rational primes by the appropriate
sum over prime ideals of $K$ and 
noting that the degree $1$ primes comprise a density $1$ subset of the set 
of all prime ideals of $K$, we 
are free to ``ignore" the condition on the degree.  Note that this trick will not work if one wants 
to count degree $2$ primes satisfying congruence conditions.  

We will need the following result in order to control the error incurred by invoking the 
Chebotar\"ev Density Theorem to estimate sums of the form
\begin{equation*}
\theta_K(x;1,q,a):=\sum_{\substack{\N\p\le x\\ \deg\p=1\\ \N\p\equiv a\pmod q}}\log\N\p.
\end{equation*} 
\begin{thm}\label{bdh gen}
For any $M>0$,
\begin{equation*}
\sum_{q\le Q}\sum_{a\in G_{q}}\left(\theta_K(x;1,q,a)-\frac{x}{\varphi_K(q)}\right)^2
\ll xQ\log x,
\end{equation*}
provided that $x(\log x)^{-M}\le Q\le x$.
\end{thm}
\begin{rmk}
This result is a slight modification of the main result of~\cite{Smi:2009} and can be 
proved similarly with only minor alterations to the proof.  In~\cite{Smi:2009}, the 
main result is stated with $\theta_K(x;1,q,a)$ replaced by the Chebychev function
\begin{equation*}
\psi_K(x;q,a):=\sum_{\substack{\N\p^m\le x\\ \N\p^m\equiv a\pmod q}}\log\N\p.
\end{equation*}
The key observation when altering the proof is that the contribution from prime powers and 
higher degree primes is negligible.
\end{rmk}

\begin{rmk} It should also be possible to achieve an 
asymptotic version of this result along the same lines as~\cite{Smi:2010}.
\end{rmk}

\begin{proof}[Proof of Proposition~\ref{weighted avg of l-series}]
As in~\cite[p. 193]{DP:2004}, 
we introduce a parameter $U$ (to be chosen later) and begin with the identity
\begin{equation*}
L(1,\chi_{d_k(p)})=\sum_{n\ge 1}\legendre{d_k(p)}{n}\frac{1}{n}
=\sum_{n\ge 1}\legendre{d_k(p)}{n}\frac{e^{-n/U}}{n}
+O\left(\frac{|d_k(p)|^{7/32}}{U^{1/2}}\right).
\end{equation*}
Whence, if 
\begin{equation}\label{U cond 1}
U\ge x^{7/16}(\log x)^{2\eta},
\end{equation}
then substitution and interchanging sums yields
\begin{equation}\label{intro U}
A_1(x;r)=\NK\sum_{k\le 2\sqrt x}\frac{1}{k}\sum_{n\ge 1}\frac{e^{-n/U}}{n}
\sum_{p\in\mathcal S_k(x;r)}\legendre{d_k(p)}{n}\log p+O\left(\frac{x}{(\log x)^\eta}\right).
\end{equation}
We now introduce another parameter $V$ and observe that contribution to~\eqref{intro U} from the 
``large" values of $k$ is
\begin{equation*}
\sum_{V<k\le 2\sqrt x}\frac{1}{k}\sum_{n\ge 1}\frac{e^{-n/U}}{n}
\sum_{p\in\mathcal S_k(x;r)}\legendre{d_k(p)}{n}\log p
\ll (x\log x)V^{-2}\log U
\ll \frac{x}{(\log x)^\eta}
\end{equation*}
if
\begin{align}
V&\ge (\log x)^{(\eta +2)/2},\label{V cond}\\
U&\le x.\label{U cond 2}
\end{align}
We also observe that for $U$ satisfying~\eqref{U cond 1}, the contribution from the 
``large" $n$ is
\begin{equation}
\sum_{k\le V}\frac{1}{k}\sum_{n\ge U\log U}\frac{e^{-n/U}}{n}
\sum_{p\in\mathcal S_k(x;r)}\legendre{d_k(p)}{n}\log p
\ll\frac{\log x}{U\log U}
\sum_{k\le V}\frac{1}{k}\sum_{\substack{m\le x\\ k^2|(4m-r^2)}}1
\ll\frac{x}{(\log x)^\eta}.
\end{equation}
Therefore, we have
\begin{equation}\label{large n and k removed}
A_1(x;r)=\NK\sum_{k\le V}\frac{1}{k}\sum_{n\le U\log U}\frac{e^{-n/U}}{n}
\sum_{p\in\mathcal S_k(x;r)}\legendre{d_k(p)}{n}\log p+O\left(\frac{x}{(\log x)^\eta}\right).
\end{equation}

Now recall the notation of page~\pageref{prime sets}.
In particular, recall that the definition of $\mathscr P_r(x)$ explicitly excludes any 
prime ideals $\p$ which ramify in the fixed Galois extension $K(\zeta_{\MK})/K$.
If $\p$ is a prime of $K$ that does not ramify in $K(\zeta_{\MK})$,
then it follows (by calculating the Frobenius at $\p$ and applying the map~\eqref{natural map}) 
that $\N\p\equiv b\pmod{\MK}$ for some $b\in G_{\MK}$.  
For each pair $n,k$, we regroup the terms of the innermost sum in~\eqref{large n and k removed} 
to see that
\begin{equation}\label{regroup terms}
\begin{split}
\NK\sum_{p\in\mathcal S_k(x;r)}\legendre{d_k(p)}{n}\log p
&=\sum_{\substack{a\in(\zz/4n\zz)\\ a\equiv 0,1\pmod 4}}\legendre{a}{n}
\sum_{\substack{p\in\mathcal S_k(x;r)\\ d_k(p)\equiv a\pmod{4n}}}\NK\log p\\
&=\sum_{\substack{a\in(\zz/4n\zz)\\ a\equiv 0,1\pmod 4}}\legendre{a}{n}
\sum_{b\in G_{\MK}}
\sum_{\substack{p\in\mathcal S_k(x;r)\\ d_k(p)\equiv a\pmod{4n}\\ p\equiv b\pmod\MK}}\NK\log p\\
&=\sum_{\substack{a\in(\zz/4n\zz)\\ a\equiv 0,1\pmod 4\\ 4|(r^2-ak^2)}}\legendre{a}{n}
\sum_{b\in G_{\MK}}
\sum_{\substack{\p\in\mathscr P_r(x)\\ \N\p\equiv\frac{r^2-ak^2}{4}\pmod{nk^2}\\ \N\p\equiv b\pmod\MK}}\log\N\p.
\end{split}
\end{equation}
Note that if $4|(r^2-ak^2)$ and $(r^2-ak^2)/4$ is not coprime to $nk^2$, then there can be at most 
finitely many degree one prime ideals $\p$ satisfying the two conditions on the innermost sum 
in the last line of~\eqref{regroup terms}.  Furthermore, this can only happen when the greatest 
common divisor of $nk^2$ and the least positive residue of $(r^2-ak^2)/4$ is itself a prime,
say $\ell$, and $\N\p=\ell$.  Now, if $\ell$ is an odd prime dividing $k$, 
then $\ell^2$ divides both $(r^2-ak^2)/4$ and $nk^2$.
Thus, this situation can only arise from $2$ and those primes dividing $n$.  
Whence the last line of~\eqref{regroup terms} is equal to
\begin{equation}\label{remove finitely many primes}
\sum_{\substack{a\in(\zz/4n\zz)\\ a\equiv 0,1\pmod 4\\ (r^2-ak^2,4nk^2)=4}}\legendre{a}{n}
\sum_{b\in G_{\MK}}
\sum_{\substack{\p\in\mathscr P_r(x)\\ \N\p\equiv\frac{r^2-ak^2}{4}\pmod{nk^2}\\ \N\p\equiv b\pmod\MK}}\log\N\p+O\left(\sum_{\substack{\ell|n\\ \ell\text{ prime}}}\log\ell\right).
\end{equation}
Now, we interchange the outer two sums and note that the two conditions on the innermost 
sum are contradictory unless $4b\equiv r^2-ak^2\pmod{(4\MK,4nk^2)}$.  Therefore, 
the main term of~\eqref{remove finitely many primes} is equal to 
\begin{equation}\label{ensure conditions are consistent}
\sum_{b\in G_{\MK}}
\sum_{\substack{
	a\in(\zz/4n\zz)\\ 
	a\equiv 0,1\pmod 4\\ 
	(r^2-ak^2,4nk^2)=4\\ 
	4b\equiv r^2-ak^2\pmod{(4\MK,4nk^2)}
	}}
\legendre{a}{n}
\sum_{\substack{\N\p\le x\\ \deg\p=1\\ \N\p\equiv\frac{r^2-ak^2}{4}\pmod{nk^2}\\ \N\p\equiv b\pmod\MK}}\log\N\p+O(1/k^2),
\end{equation}
where the big-$O$ accounts for the prime ideals with norm less than $B(r)$
and those dividing the different of $K/\qq$.
See the definitions on page~\pageref{prime sets}.
Now, the two conditions on the innermost sum are equivalent via the Chinese Remainder 
Theorem to a single condition modulo the least common multiple $[\MK,nk^2]$.
Therefore,~\eqref{large n and k removed},~\eqref{regroup terms},
~\eqref{remove finitely many primes},~\eqref{ensure conditions are consistent}, 
and the Chebotar\"ev Density Theorem applied the prime ideals of $K$ satisfying the 
conditions of the innermost sum of~\eqref{ensure conditions are consistent}
 imply that
\begin{equation}\label{apply Chebotarev}
\begin{split}
A_1(x;r)=&x\sum_{b\in G_{\MK}}
\sum_{\substack{n\le U\log U,\\ k\le V}}
\frac{e^{-n/U}c_k^{r,b,\MK}(n)}{nk\varphi_K([\MK,nk^2])}
+O\left(
\frac{x}{(\log x)^\eta}
\right)\\
&+O\left(
\sum_{k\le V}\frac{1}{k}\sum_{n\le U\log U}\sum_{h\in G_{[\MK,nk^2]}}
\frac{e^{-n/U}}{n}\left|E_K(x;1,[\MK,nk^2],h)\right|
\right),
\end{split}
\end{equation}
where $c_k^{r,b,\MK}(n)$ is the function defined by equation~\eqref{c fct defn} on 
page~\pageref{c fct defn}, and for $h\in G_q$
\begin{equation*}
E_K(x;1,q,h):=\theta_K(x;1,q,h)-\frac{x}{\varphi_K(q)}.
\end{equation*}

Facts from Galois theory imply that $G_{[\MK,nk^2]}$ is a quotient of 
$G_{\MK nk^2}$; whence, via the triangle inequality, we have
\begin{equation*}
\sum_{h\in G_{[\MK,nk^2]}}\left|E_K(x;1,[\MK,nk^2],h)\right|
\le\sum_{h\in G_{\MK nk^2}}\left|E_K(x;1,\MK nk^2,h)\right|.
\end{equation*}
We now choose
\begin{align*}
U&=\frac{x}{(\log x)^{5\eta+15}},\\
V&=(\log x)^{(\eta+3)/2}
\end{align*}
and note that this choice is in accordance with~\eqref{U cond 1},~\eqref{V cond},
and~\eqref{U cond 2}.
Thus, by the Cauchy-Schwarz inequality and Theorem~\ref{bdh gen}, we have
\begin{equation*}
\begin{split}
\sum_{k\le V}\frac{1}{k}&\left[\sum_{n\le U\log U}\sum_{h\in G_{[\MK,nk^2]}}
\frac{e^{-n/U}}{n}\left|E_K(x;1,[\MK,nk^2],h)\right|\right]\\
&\le\sum_{k\le V}\frac{1}{k}
\left[
\sum_{n\le U\log U}\frac{\varphi_K(\MK nk^2)}{n^2}
\right]^{1/2}
\left[
\sum_{n\le U\log U}\sum_{h\in G_{\MK nk^2}}E_K(x;1,\MK nk^2,h)^2
\right]^{1/2}\\
&\ll V\sqrt{\log U}\left[
\sum_{q\le\MK V^2U\log U}\sum_{h\in G_q}E_K(x;1,q,h)^2
\right]^{1/2}\\
&\ll V\sqrt{\log U}\sqrt{xV^2U\log U\log x}\\
&\ll\frac{x}{(\log x)^\eta}
\end{split}
\end{equation*}
since $x(\log x)^{-M}\le V^2U\log U\le x$, say with $M=4\eta+11$.
Therefore, equation~\eqref{apply Chebotarev} becomes
\begin{equation}\label{simplify after bdh}
A_1(x;r)=
x\sum_{b\in G_{\MK}}
\sum_{\substack{n\le U\log U,\\ k\le V}}
\frac{e^{-n/U}c_k^{r,b,\MK}(n)}{nk\varphi_K([\MK,nk^2])}
+O\left(
\frac{x}{(\log x)^\eta}
\right)
\end{equation}

Recall the definition of $\mathcal A$ from the first paragraph of Section~\ref{avg order const sect},
and observe that since $\MK|[\MK,nk^2]$, we have 
\begin{equation*}
\mathcal A=\qq^{\text{cyc}}\cap K=\qq(\zeta_{\MK})\cap K\subseteq\qq(\zeta_{[\MK,nk^2]}).
\end{equation*}
Thus, we have the isomorphism
$G_{[\MK,nk^2]}\isom\Gal(\qq(\zeta_{[\MK,nk^2]})/\mathcal A)$, and recalling the definition of 
$\NA$, we have the identity
\begin{equation*}
\varphi([\MK,nk^2])=\NA\varphi_K([\MK,nk^2]).
\end{equation*}
Hence, equation~\eqref{simplify after bdh} becomes
\begin{equation}
A_1(x;r)=
x\NA\sum_{b\in G_{\MK}}
\sum_{\substack{n\le U\log U,\\ k\le V}}
\frac{e^{-n/U}c_k^{r,b,\MK}(n)}{nk\varphi([\MK,nk^2])}
+O\left(
\frac{x}{(\log x)^\eta}
\right).
\end{equation}

The final step of the proof is to show that
\begin{equation*}
\sum_{b\in G_{\MK}}
\sum_{\substack{n\le U\log U,\\ k\le V}}
\frac{e^{-n/U}c_k^{r,b,\MK}(n)}{nk\varphi([\MK,nk^2])}
=
\sum_{b\in G_{\MK}}
\sum_{\substack{n\ge 1,\\ k\ge 1}}
\frac{c_k^{r,b,\MK}(n)}{nk\varphi([\MK,nk^2])}
+O\left(\frac{1}{(\log x)^\eta}\right)
\end{equation*}
which, as a side effect, demonstrates the absolute convergence of the double infinite 
sum in~\eqref{avg const}.  This is done in a manner similar to~\cite[pp. 197-199]{DP:2004}.
\end{proof}

\bibliographystyle{plain}
\bibliography{references}
\end{document}